\documentclass[12pt]{amsart}
\usepackage{amsmath,amsthm,amsfonts,amssymb,mathrsfs}
\date{\today}

\usepackage{color}
\input xymatrix
\xyoption{all}

\usepackage{hyperref}

\newtheorem{theorem}{Theorem}
\newtheorem*{theorem*}{Main Theorem}

\newtheorem{question}{Question}
\newtheorem{proposition}{Proposition}
\newtheorem{corollary}{Corollary}
\newtheorem*{question*}{Question}
\theoremstyle{definition}
\newtheorem{example}{Example}
\newtheorem{remark}{Remark}

\def\N{\mathbb N}
\def\op{\operatorname}

\begin{document}

\title[On universal objects in the class of graph inverse semigroups]{On universal objects in the class of graph inverse semigroups}

\author[S.~Bardyla]{Serhii~Bardyla}
\address{Faculty of Applied Mathematics and Informatics, National University of Lviv,
Universytetska 1, Lviv, 79000, Ukraine}
\email{sbardyla@yahoo.com}



\begin{abstract}
In this paper we show that polycyclic monoids are universal objects in the class of graph inverse semigroups.
In particular, we prove that a graph inverse semigroup $G(E)$ over a directed graph $E$ embeds into the polycyclic monoid $\mathcal{P}_{\lambda}$ where $\lambda=|G(E)|$. We show that each graph inverse semigroup $G(E)$ admits the coarsest inverse semigroup topology $\tau$. Moreover, each injective homomorphism from $(G(E),\tau)$ to the $(\mathcal{P}_{|G(E)|},\tau)$ is a topological embedding.
\end{abstract}
\maketitle

We shall follow the terminology of~\cite{Clifford-Preston-1961-1967} and
\cite{Lawson-1998}. By $|A|$ we denote the cardinality of a set $A$ and by $\omega$ we denote the first infinite cardinal.
A~semigroup $S$ is called an \emph{inverse semigroup} if every $a$
in $S$ possesses a unique inverse, i.e., if there exists a unique
element $a^{-1}$ in $S$ such that
\begin{equation*}
    aa^{-1}a=a \qquad \mbox{and} \qquad a^{-1}aa^{-1}=a^{-1}.
\end{equation*}
The map $S\to S$, $x\mapsto x^{-1}$ assigning to each element of an inverse semigroup its inverse is called the \emph{inversion}.

A {\em directed graph} $E=(E^{0},E^{1},r,s)$ consists of sets $E^{0},E^{1}$ of {\em vertices} and {\em edges}, respectively, together with functions $s,r:E^{1}\rightarrow E^{0}$ which are called {\em source} and {\em range}, respectively. In this paper we refer to directed graphs simply as ``graphs".  A path $x=e_{1}\ldots e_{n}$ in graph $E$ is a finite sequence of edges $e_{1},\ldots,e_{n}$ such that $r(e_{i})=s(e_{i+1})$ for each positive integer $i<n$. We extend the source and range functions $s$ and $r$ on the set $\operatorname{Path}(E)$ of all pathes in graph $E$ as follows: for each $x=e_{1}\ldots e_{n}\in \operatorname{Path}(E)$ put $s(x)=s(e_{1})$ and $r(x)=r(e_{n})$. By $|x|$ we denote the length of a path $x$. We consider each vertex being a path of length zero. An edge $e$ is called a {\em loop} if $s(e)=r(e)$. A path $x$ is called a {\em cycle} if $s(x)=r(x)$ and $|x|>0$.


The bicyclic monoid ${\mathscr{C}}(p,q)$ is the semigroup with the identity $1$ generated by two elements $p$ and $q$ subject to the condition $pq=1$.


One of the generalizations of the bicyclic monoid is a polycyclic monoid.
For a non-zero cardinal $\lambda$, polycyclic monoid $\mathcal{P}_\lambda$ is the semigroup with identity and zero given by the presentation:
\begin{equation*}
    \mathcal{P}_\lambda=\left\langle \left\{p_i\right\}_{i\in\lambda}, \left\{p_i^{-1}\right\}_{i\in\lambda}\mid  p_i^{-1}p_i=1, p_j^{-1}p_i=0 \hbox{~for~} i\neq j\right\rangle.
\end{equation*}

Polycyclic monoid $\mathcal{P}_{k}$ over a finite non-zero cardinal $k$ was introduced in~\cite{Nivat-Perrot-1970}. Observe that the bicyclic semigroup with an adjoined zero is isomorphic to the polycyclic monoid $\mathcal{P}_{1}$. For a finite cardinal $k$ algebraic properties of a semigroup $\mathcal{P}_{k}$ were investigated in \cite{Lawson-2009} and \cite{Meakin-1993}.

For a given directed graph $E=(E^{0},E^{1},r,s)$ the graph inverse semigroup (or simply GIS) $G(E)$ over the graph $E$ is a semigroup with zero generated by the sets $E^{0}$, $E^{1}$ together with a set $E^{-1}=\{e^{-1}:e\in E^{1}\}$ satisfying the following relations for all $a,b\in E^{0}$ and $e,f\in E^{1}$:
\begin{itemize}
 \item [(i)]  $a\cdot b=a$ if $a=b$ and $a\cdot b=0$ if $a\neq b$;
 \item [(ii)] $s(e)\cdot e=e\cdot r(e)=e;$
 \item [(iii)] $e^{-1}\cdot s(e)=r(e)\cdot e^{-1}=e^{-1};$
 \item [(iv)] $e^{-1}\cdot f=r(e)$ if $e=f$ and $e^{-1}\cdot f=0$ if $e\neq f$.
\end{itemize}

Graph inverse semigroups are generalizations of the polycyclic monoids. In particular, for every non-zero cardinal $\lambda$ the polycyclic monoid $\mathcal{P}_{\lambda}$ is isomorphic to the graph inverse semigroup $G(E)$ over the graph $E$ which consists of one vertex and $\lambda$ distinct loops.

According to~\cite[Chapter~3.1]{Jones-2011}, each non-zero element of a graph inverse semigroup $G(E)$ is of the form  $ab^{-1}$ where $a,b\in \operatorname{Path}(E)\hbox{ and } r(a)=r(b)$. A semigroup operation in $G(E)$ is defined by the following formula:
\begin{equation*}
\begin{split}
  &  ab^{-1}\cdot cd^{-1}=
    \left\{
      \begin{array}{ccl}
        ac_{1}d^{-1}, & \hbox{if~~} c=bc_{1} & \hbox{for some~} c_1\in \operatorname{Path}(E);\\
        a(db_{1})^{-1},   & \hbox{if~~} b=cb_{1} & \hbox{for some~} b_1\in \operatorname{Path}(E);\\
        0,              & \hbox{otherwise},
      \end{array}
    \right.
    \end{split}
\end{equation*}
  and $ab^{-1}\cdot 0=0\cdot ab^{-1}=0\cdot 0=0.$

Simply verifications show that $G(E)$ is an inverse semigroup, moreover, $(uv^{-1})^{-1}=vu^{-1}$.

Graph inverse semigroups play an important role in the study of rings and $C^{*}$-algebras (see \cite{Abrams-2005,Ara-2007,Cuntz-1980,Kumjian-1998,Paterson-1999}).
Algebraic properties of graph inverse semigroups were studied in~\cite{Amal-2016, Bardyla-2018(1), Jones-2011, Jones-Lawson-2014, Lawson-2009, Mesyan-2016}.

In this paper we show that polycyclic monoids are universal objects in the class of graph inverse semigroups.
In particular, we prove that a graph inverse semigroup $G(E)$ over a directed graph $E$ embeds into the polycyclic monoid $\mathcal{P}_{\lambda}$ where $\lambda=|G(E)|$. We show that each graph inverse semigroup $G(E)$ admits the coarsest inverse semigroup topology $\tau$. Moreover, each injective homomorphism from $(G(E),\tau)$ to the $(\mathcal{P}_{|G(E)|},\tau)$ is a topological embedding.

\section{Main result}
For a graph $E$ put $\lambda_E=\max\{|E^0|,|E^1|,\omega\}$. A graph $E$ is called {\em countable} if $\lambda_E=\omega$.
\begin{theorem}\label{theorem1}
For an arbitrary graph $E$ the graph inverse semigroup $G(E)$ embeds into the polycyclic monoid $\mathcal{P}_{\lambda_E}$.
\end{theorem}

\begin{proof}
Fix an arbitrary graph $E$.
Let $G=G^+\cup G^-=\{p_{\alpha}\}_{\alpha\in\lambda_E}\cup\{p_{\alpha}^{-1}\}_{\alpha\in\lambda_E}$ be a set of generators of $\mathcal{P}_{\lambda_E}$. By "$\cdot$" we denote the semigroup operation in $\mathcal{P}_{\lambda_E}$.

Fix arbitrary injections $g:E^{0}\rightarrow G^+$ and $h:E^{1}\rightarrow G^+$.
For each vertex $a\in E^{0}$ put $$F(a)=g(a)g(a)^{-1}.$$
For each edge $e\in E^{1}$ put
 $$F(e)=g(s(e))h(e)g(r(e))^{-1}\hbox{ and } F(e^{-1})=F(e)^{-1}=g(r(e))h(e)^{-1}g(s(e))^{-1}.$$
We extend the map $F$ on the semigroup $G(E)$ by the following way: for each non-zero element $uv^{-1}=e_1e_2\ldots e_n(f_1f_2\ldots f_m)^{-1}\in G(E)$ put $$F(uv^{-1})=F(e_1)\cdot F(e_2)\cdots F(e_n)\cdot F(f_m^{-1})\cdot F(f_{m-1}^{-1}) \cdots F(f_1^{-1})\qquad \hbox{and} \qquad F(0)=0.$$

Firstly we show that the map $F:G(E)\rightarrow \mathcal{P}_{\lambda_E}$ is injective. Fix an arbitrary non-zero element $$uv^{-1}=e_1e_2\ldots e_n(f_1f_2\ldots f_m)^{-1}\in G(E).$$
Since $u,v\in\operatorname{Path}(E)$ and $r(u)=r(v)$ we obtain that $g(r(e_n))=g(r(f_m))$ and  $g(r(e_i))=g(s(e_{i+1}))$, $g(r(f_j))=g(s(f_{j+1}))$ for each positive integers $i<n$ and $j<m$.
Then
\begin{equation*}
\begin{split}
& F(e_1e_2\ldots e_n(f_1f_2\ldots f_m)^{-1})=F(e_1)\cdot F(e_2)\cdots F(e_n)\cdot F(f_m^{-1})\cdot F(f_{m-1}^{-1})\cdots F(f_1^{-1})=\\
& =g(s(e_1))h(e_1)\big[g(r(e_1))^{-1}\cdot g(s(e_2))\big]h(e_2)\big[g(r(e_2))^{-1}\cdot g(s(e_3))\big]h(e_3)g(r(e_3))^{-1}\cdots \\
& \cdots g(s(e_n))h(e_n)\big[g(r(e_n))^{-1}\cdot g(r(f_m))\big]h(f_m)^{-1}\big[g(s(f_m))^{-1}\cdot g(r(f_{m-1}))\big]h(f_{m-1})^{-1}g(s(f_{m-1}))^{-1}\cdots\\
& \cdots g(r(f_2))h(f_2)^{-1}\big[g(s(f_2))^{-1}\cdot g(r(f_1))\big]h(f_1)^{-1}g(s(f_1))^{-1}=g(s(e_1))\cdot h(e_1)\cdot 1\cdot h(e_2) \cdot 1\cdots\\
&\cdots h(e_n)\cdot 1 \cdot h(f_m)^{-1}\cdot 1 \cdot h(f_{m-1})^{-1}\cdot 1\cdots h(f_2)^{-1}\cdot 1 \cdot h(f_1)^{-1}\cdot g(s(f_1))^{-1}=\\
& = g(s(e_1))\cdot h(e_1)\cdot h(e_2)\cdots h(e_n)\cdot h(f_m)^{-1} \cdot h(f_{m-1})^{-1}\cdots h(f_1)^{-1}\cdot g(s(f_1))^{-1}.\\
\end{split}
\end{equation*}
Since $h$ and $g$ are injections we see that $F(u_1v_1^{-1})=F(u_2v_2^{-1})$ if and only if $u_1=u_2$ and $v_1=v_2$, for each non-zero elements $u_1v_1^{-1}$ and $u_2v_2^{-1}$ of $G(E)$. Also, $F(uv^{-1})\neq 0$ for each non-zero element $uv^{-1}\in G(E)$. Hence the map $F$ is an injection.

Now we are going to show that the map $F$ is a homomorphism. Observe that for each element $uv^{-1}\in G(E)$, $F(uv^{-1}\cdot 0)=F(uv^{-1})\cdot F(0)=0$ and $F(0 \cdot uv^{-1})=F(0)\cdot F(uv^{-1})=0$. Fix two non-zero elements $ab^{-1}=a_1\ldots a_n(b_1\ldots b_m)^{-1}$ and $cd^{-1}=c_1\ldots c_k(d_1\ldots d_t)^{-1}$ of the semigroup $G(E)$. There are four cases to consider:
\begin{itemize}
\item[$(1)$] $ab^{-1}\cdot cd^{-1}=aud^{-1}$, i.e., $c=bu$ for some path $u$ such that $|u|>0$;
\item[$(2)$] $ab^{-1}\cdot cd^{-1}=a(dv)^{-1}$, i.e., $b=cv$ for some path $v$ such that $|v|>0$;
\item[$(3)$] $ab^{-1}\cdot cd^{-1}=ad^{-1}$, i.e., $b=c$;
\item[$(4)$] $ab^{-1}\cdot cd^{-1}=0$.
\end{itemize}

Consider case $(1)$. Assume that $c=bu=b_1\ldots b_mc_{m+1}\ldots c_k$. Observe that $g(s(b_{i+1}))=g(r(b_i))$ for each $i<m$.
Then
\begin{equation*}
\begin{split}
& F(ab^{-1})\cdot F(cd^{-1})=F(a)\cdot F(b_m)^{-1}\cdots F(b_1)^{-1}\cdot F(b_1)\cdots F(b_m)\cdot F(c_{m+1})\cdots F(c_{k})\cdot F(d^{-1})= \\
&= F(a)\cdot g(r(b_m))h(b_m)^{-1}g(s(b_m))^{-1}\cdots g(r(b_2))h(b_2)^{-1}g(s(b_2))^{-1}\cdot  g(r(b_1))\big[h(b_1)^{-1}g(s(b_1))^{-1}\cdot\\
&\cdot g(s(b_1))h(b_1)\big]g(r(b_1))^{-1}\cdot g(s(b_2))h(b_2)g(r(b_2))^{-1} \cdots g(s(b_m))h(b_m)g(r(b_m))^{-1}\cdot F(c_{m+1})\cdots \\
&\cdots F(c_{k})\cdot F(d^{-1})= F(a)\cdot g(r(b_m))h(b_m)^{-1}g(s(b_m))^{-1}\cdots g(r(b_2))h(b_2)^{-1}g(s(b_2))^{-1}\cdot\big[g(r(b_1))\cdot \\
&\cdot g(r(b_1))^{-1}\cdot g(s(b_2))\big]h(b_2)g(r(b_2))^{-1} \cdots g(s(b_m))h(b_m)g(r(b_m))^{-1}\cdot F(c_{m+1})\cdots F(c_{k})\cdot F(d^{-1})=\\
&= F(a)\cdot g(r(b_m))h(b_m)^{-1}g(s(b_m))^{-1}\cdots g(r(b_2))\big[h(b_2)^{-1}g(s(b_2))^{-1}\cdot g(s(b_2))h(b_2)\big]g(r(b_2))^{-1} \cdots\\
& \cdots g(s(b_m))h(b_m)g(r(b_m))^{-1}\cdot F(c_{m+1})\cdots F(c_{k})\cdot F(d^{-1})=F(a)\cdot g(r(b_m))h(b_m)^{-1}g(s(b_m))^{-1}\cdots\\
&\big[ g(r(b_2))\cdot g(r(b_2))^{-1}\cdot g(s(b_3))\big]\cdots g(s(b_m))h(b_m)g(r(b_m))^{-1}\cdot F(c_{m+1})\cdots F(c_{k})\cdot F(d^{-1})=\ldots\\
& =F(a)\cdot F(c_{m+1})\cdots F(c_{k})\cdot F(d^{-1})=F(aud^{-1})=F(ab^{-1}\cdot cd^{-1}).\\
\end{split}
\end{equation*}
Cases $(2)$ and $(3)$ are similar to case $(1)$.
Consider case $(4)$. Since $ab^{-1}\cdot cd^{-1}=0$ there exists a positive integer $n_{0}\leq\min\{m,k\}$ such that $b_{i}=c_i$ for each positive integer $i<n_0$ and $b_{n_{0}}\neq c_{n_0}$.
Hence
\begin{equation*}
\begin{split}
& F(ab^{-1})\cdot F(cd^{-1})=F(a)\cdot F(b_m)^{-1}\cdots F(b_{n_{0}})^{-1}\cdot \big[F(b_{n_{0}-1})^{-1}\cdots F(b_1)^{-1}\cdot F(b_1)\cdots F(b_{n_{0}-1})\big]\cdot\\
& \cdot F(c_{n_0})\cdots F(c_{k})\cdot F(d)^{-1}= F(a)\cdot F(b_m)^{-1}\cdots F(b_{n_{0}})^{-1}\cdot F(c_{n_{0}})\cdots F(c_k)\cdot F(d)^{-1}=\\
& = F(a)\cdot F(b_m)^{-1}\cdots F(b_{n_{0}+1})^{-1}\cdot g(r(b_{n_{0}}))\big[h(b_{n_0})^{-1}g(s(b_{n_{0}}))^{-1} \cdot g(s(c_{n_{0}}))h(c_{n_0})\big]g(r(c_{n_{0}}))^{-1}\cdot \\
&\cdot F(c_{n_{0}+1})\cdots F(c_k)\cdot F(d)^{-1}=0=F(0),\\
\end{split}
\end{equation*}
because $b_{n_0}\neq c_{n_0}$ and hence even if $s(b_{n_0})=s(c_{n_0})$ (which implies that $g(s(b_{n_{0}}))^{-1} \cdot g(s(c_{n_{0}}))=1$) we have that $h(b_{n_{0}})\neq h(c_{n_{0}})$ which yields that $h(b_{n_{0}})^{-1}\cdot h(c_{n_{0}})=0$.
Hence the map $F$ is an embedding of the semigroup $G(E)$ into the polycyclic monoid $\mathcal{P}_{\lambda_{E}}$.
\end{proof}

Since the polycyclic monoid $\mathcal{P}_{\omega}$ embeds into the polycyclic monoid $\mathcal{P}_2$ (see \cite[Chapter 9.3, Proposition 6]{Lawson-1998}) Theorem~\ref{theorem1} implies the following:

\begin{theorem}\label{th2}
Each graph inverse semigroup $G(E)$ over a countable graph $E$ embeds into the polycyclic monoid $\mathcal{P}_2$.
\end{theorem}

Since  for each inverse semigroup $G(E)$, $\lambda_E=\omega$ if $|G(E)|\leq\omega$ and $|G(E)|=\lambda_E$ if $|G(E)|>\omega$
Theorems~\ref{theorem1} and~\ref{th2} imply the following:

\begin{corollary}
Let $G(E)$ be the graph inverse semigroup over an arbitrary graph $E$. Then the following conditions hold:
\begin{itemize}
\item if $|G(E)|\leq\omega$ then $G(E)$ embeds into the polycyclic monoid $\mathcal{P}_2$;
\item if $|G(E)|=\lambda>\omega$ then $G(E)$ embeds into the polycyclic monoid $\mathcal{P}_{\lambda}$.
\end{itemize}
\end{corollary}

\begin{remark}
However, there exists an inverse subsemigroup of the polycyclic monoid $\mathcal{P}_2$ which is not isomorphic to any graph inverse semigroup. An example of such subsemigroup is a Gauge monoid (see \cite[Chapter 2.3]{Jones-2011}). Moreover, an inverse monoid $S$ is isomorphic to some graph inverse semigroup $G(E)$ if and only if $S$ is isomorphic to the polycyclic monoid $\mathcal{P}_{\lambda}$ for some cardinal $\lambda$. Indeed, if a graph $E$ contains at least two vertices then semigroup $G(E)$ does not contain a unit and if graph $E$ contains only one vertex then $G(E)$ is isomorphic to the monoid $\mathcal{P}_{\lambda}$, where $\lambda=|E^1|$ (here we agree that monoid $\mathcal{P}_{0}$ is isomorphic to the semilattice $(\{0,1\},\min)$).
\end{remark}

\section{Topological versions of Theorem~\ref{theorem1}}

In this section we investigate embeddings of graph inverse semigroups into topological inverse polycyclic monoids. After some preparatory results concerning topologizing of graph inverse semigroups, we construct (in canonical way) the coarsest inverse semigroup topology $\tau$ on each graph inverse semigroup $G(E)$. Moreover, each injective homomorphism $f:(G(E),\tau)\rightarrow (\mathcal{P}_{\lambda},\tau)$ is a topological embedding.

A {\em topological (inverse) semigroup} is a Hausdorff topological space together with a continuous semigroup operation (and an~inversion, respectively).  If $S$ is a semigroup (an inverse semigroup) and $\tau$ is a topology on $S$ such that $(S,\tau)$ is a topological (inverse) semigroup, then we
shall call $\tau$ a ({\em inverse}) {\em semigroup topology} on $S$.
A {\em semitopological semigroup} is a Hausdorff topological space together with a separately continuous semigroup operation.

Topological and semitopological graph inverse semigroups were investigated in~\cite{Bardyla-2018,Bardyla-2017,Bardyla-2016(1),BardGut-2016(1),BardGut-2016(2),Gutik-2015,Mesyan-Mitchell-Morayne-Peresse-2013}.

Let $E$ be an arbitrary graph. A path $a\in\operatorname{Path}(E)$ is called a {\em prefix} of a path $b\in\operatorname{Path}(E)$ if $b=ac$ for some path $c\in\operatorname{Path}(E)$. By $\leq$ we denote a partial order on $\operatorname{Path}(E)$ which is defined as follows: for each elements $a,b\in \op{Path}(E)$, $a\leq b$ iff $b$ is a prefix of $a$. A set which is endowed with a partial order is called {\em poset}. For each element $x$ of a poset $X$ denote ${\downarrow}x=\{y\in X\mid x\leq y\}$. For a subset $A\in X$ put ${\downarrow}A=\cup_{x\in A}{\downarrow}x$.  A subset $A$ of a poset $(X,\leq)$ is called an {\em ideal} if ${\downarrow}A=A$.

A family $\mathcal{F}$ of subsets of a set $X$ is called a {\em filter} if it satisfies the following conditions:
\begin{itemize}
\item[$(1)$] $\emptyset\notin \mathcal{F}$;
\item[$(2)$] If $A\in \mathcal{F}$ and $A\subset B$ then $B\in \mathcal{F}$;
\item[$(3)$] If $A,B\in \mathcal{F}$ then $A\cap B\in\mathcal{F}$.
\end{itemize}
A family $\mathcal{B}$ is called a {\em base} of a filter $\mathcal{F}$ if for each element $A\in\mathcal{F}$ there exists element $B\in\mathcal{B}$ such that $B\subset A$. A filter $\mathcal{F}$ is called {\em free} if $\cap_{F\in\mathcal{F}}=\emptyset$.

%

Let $E$ be an arbitrary graph and $\mathcal{F}$ be a free filter on the set $\operatorname{Path}(E)$.
The filter $\mathcal{F}$ generates the topology $\tau_{\mathcal{F}}$ on GIS $G(E)$ which is defined as follows: each non-zero element of $G(E)$ is  isolated (by~\cite[Theorem~4]{Bardyla-2018}, this condition is necessary if we want $(G(E),\tau_{\mathcal{F}})$ to be a semitopological semigroup) and $\mathcal{B}_{\mathcal{F}}(0)=\{U_{F}(0): F\in\mathcal{F}\}$ is a base of the topology $\tau_{\mathcal{F}}$ at zero $0\in G(E)$ where $U_{F}(0)=\{ab^{-1}\in G(E): a,b\in F\}\cup\{0\}.$

By $\mathcal{F}_{cf}$ we denote a filter of cofinite subsets of $\op{Path}(E)$.
A filter $\mathcal{F}$ on the set $\op{Path}(E)$ is called {\em topological} if $\mathcal{F}$ satisfies the following conditions:
\begin{itemize}
\item[$(i)$] Let $F\in \mathcal{F}$ and $a,b\in\op{Path}(E)$ such that $r(a)=r(b)$. Then the set $F_1=F\setminus\{bk\mid k\in\op{Path}(E)\hbox{ and }ak\notin F\}$ belongs to $\mathcal{F}$.
\item[$(ii)$] $\mathcal{F}_{cf}\subset \mathcal{F}$;
\item[$(iii)$] $\mathcal{F}$ has a base which consists of ideals of $(\op{Path}(E),\leq)$.
\end{itemize}


\begin{proposition}\label{prop0}
Let $E$ be an arbitrary graph and $\mathcal{F}$ be a topological filter on the set $\operatorname{Path}(E)$.
Then $(G(E),\tau_{\mathcal{F}})$ is a topological inverse semigroup.
\end{proposition}

\begin{proof}
Fix an arbitrary non-zero element $uv^{-1}\in G(E)$. By condition $(ii)$, the set $F=\op{Path}(E)\setminus \{u,v\}$ belongs to $\mathcal{F}$. Then $uv^{-1}\notin U_{F}(0)$. Hence a topological space $(G(E),\tau_{\mathcal{F}})$ is Hausdorff.

Observe that $\left(U_{F}(0)\right)^{-1}=U_{F}(0)$ for any $F\in\mathcal{F}$. Hence the inversion is continuous in $(G(E),\tau_{\mathcal{F}})$.

To prove the continuity of the semigroup operation in $(G(E),\tau_{\mathcal{F}})$ we need to consider the following three cases:
\begin{itemize}
\item[$(1)$] $ab^{-1}\cdot 0=0$;
\item[$(2)$] $0\cdot ab^{-1}=0$;
\item[$(3)$] $0\cdot 0=0$.
\end{itemize}

Consider case $(1)$.
Fix an arbitrary element $ab^{-1}\in G(E)$ and any basic open neighbourhood $U_{F}(0)$ of 0. Let $P_b$ be the set of all prefixes of the path $b$. Put $H=F\setminus (P_b\cup\{bk\in \operatorname{Path}(E): ak\notin F\})$. Since the filter $\mathcal{F}$ satisfies condition $(i)$, the set $F\setminus\{bk\in \operatorname{Path}(E): ak\notin F\}\in\mathcal{F}$.  Obviously, the set $P_{b}$ is finite. By condition $(ii)$, the set $F\setminus P_b$ belongs to $\mathcal{F}$. Hence $H\in\mathcal{F}$.

We claim that $ab^{-1}\cdot U_{H}(0)\subseteq U_{F}(0)$.
Indeed, fix an arbitrary element $cd^{-1}\in U_{H}(0)$. If $ab^{-1}\cdot cd^{-1}=0$ then there is nothing to prove. Suppose that $ab^{-1}\cdot cd^{-1}\neq 0$. The choice of the neighbourhood $U_{H}(0)$ implies that $ab^{-1}\cdot cd^{-1}=ac_1d^{-1}$, i.e., $c=bc_1$ for some path $c_1\in\operatorname{Path}(E)$. Observe that $d\in F$ and the definition of the set $H$ implies that $ac_1\in F$. Hence $ac_1d^{-1}\in U_{F}(0)$.

Consider case $(2)$.
Put $G=F\setminus (P_a\cup\{ak\in \operatorname{Path}(E)\colon bk\notin F\})$. Similar arguments imply that $G\in\mathcal{F}$ and $U_{G}(0)\cdot ab^{-1}\subseteq U_{F}(0)$.

Consider case $(3)$. Fix an arbitrary open neighbourhood $U_F(0)$ of $0$.
Let $T$ be any element of $\mathcal{F}$ such that $T\subset F$ and $T$ is an ideal of $(\operatorname{Path}(E),\leq)$.
We claim that $U_{T}(0)\cdot U_{T}(0)\subseteq U_{F}(0)$. Indeed, fix an arbitrary elements $ab^{-1}\in U_{T}(0)$ and $cd^{-1}\in U_{T}(0)$. If $ab^{-1}\cdot cd^{-1}=0$ then there is nothing to prove. Suppose that $ab^{-1}\cdot cd^{-1}\neq 0$. Then either
$$ab^{-1}\cdot cd^{-1}=ac_1d^{-1}\qquad \hbox{or}\qquad ab^{-1}\cdot cd^{-1}=a(db_1)^{-1}.$$
Since $T$ is an ideal of $(\operatorname{Path}(E),\leq)$ we obtain that $ac_1\in T$ and $db_1\in T$ which implies that $U_{T}(0)\cdot U_{T}(0)\subseteq U_{T}(0)\subset U_F(0)$.

Hence $(G(E),\tau_{\mathcal{F}})$ is a topological inverse semigroup.
\end{proof}



The following example shows that there exists a graph inverse semigroup $G(E)$ which admits inverse semigroup topology generated by a filter on the set $\op{Path}(E)$ which does not have a base consisting of ideals of $(\op{Path}(E),\leq)$.

\begin{example}
Let $E$ be a graph which is depicted below.

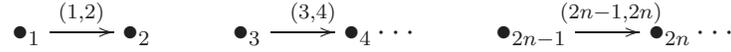
\centerline{
\xymatrix{
    \bullet_{1} \ar[r]^{(1,2)}&\bullet_{2} &\bullet_{3} \ar[r]^{(3,4)}&\bullet_{4} \cdots &\bullet_{2n-1} \ar[r]^{(2n-1,2n)}&\bullet_{2n} \cdots
    }
}

We enumerate vertices of graph $E$ with positive integers and identify each edge $x$ with a pair of positive integers $(2n$$-1,2n)$ where $s(x)=2n$$-1$ and $r(x)=2n$.
Clearly that $\op{Path}(E)=E^0\cup E^1$ where $E^0$ is the set of all vertices of graph $E$ and $E^1$ is the set of all edges of graph $E$. Let $\mathcal{F}$ be the filter on $\op{Path}(E)$ which base consists of cofinite subsets of $E_0$. Since the filter $\mathcal{F}$ contains the filter $\mathcal{F}_{cf}$ the space $(G(E),\tau_{\mathcal{F}})$ is Hausdorff. The continuity of the semigroup operation in $(G(E),\tau_{\mathcal{F}})$ follows from the following equations:
\begin{itemize}
\item[] $ab^{-1}\cdot U_F(0)\subseteq\{0\}$ where $F=E^0\setminus s(b)$;
\item[] $U_F(0)\cdot ab^{-1}\subseteq\{0\}$ where $F=E^0\setminus s(a)$;
\item[] $U_F\cdot U_F=U_F$ for an arbitrary $F\subseteq E^0$.
\end{itemize}
Hence $(G(E),\tau_{\mathcal{F}})$ is a topological inverse semigroup.

However, the filter $\mathcal{F}$ does not admit a base which consists of ideals of $(\op{Path}(E),\leq)$. Indeed, fix an arbitrary cofinite subset $F$ of $E^0$. The largest ideal which is contained in $F$ is the set $F\setminus \{2n-1\mid n\in \N\}$ which does not belong to the filter $\mathcal{F}$.
\end{example}

Now we are going to construct two canonical examples of inverse semigroup topologies on graph inverse semigroups which are generated by topological filters.

\begin{example}\label{ex1}
For each positive integer $n$ put $U_n=\{u\in\operatorname{Path}(E):|u|>n\}$.
Let $\mathcal{F}_{\omega}$ be the filter on the set $\operatorname{Path}(E)$ generated by the base consisting of the sets $U_n$, $n\in\N$.
Simple verifications show that the filter $\mathcal{F}_{\omega}$ is topological. Hence, by Proposition~\ref{prop0}, $(G(E),\tau_{\mathcal{F}_{\omega}})$ is a topological inverse semigroup.
\end{example}

\begin{example}\label{ex2}
Let $\mathcal{F}_{cf}$ be the filter consisting of cofinite subsets of $\operatorname{Path}(E)$.
Simple verifications show that the filter $\mathcal{F}_{cf}$ is topological. Hence, by Proposition~\ref{prop0}, $(G(E),\tau_{\mathcal{F}_{cf}})$ is a topological inverse semigroup.
\end{example}

\begin{remark}
The following arguments imply that in the general case the topological spaces $(G(E),\tau_{\mathcal{F}_{\omega}})$ and $(G(E),\tau_{\mathcal{F}_{cf}})$ are not homeomorphic.
A topological space $(G(E),\tau_{\mathcal{F}_{\omega}})$ is discrete if and only if there exists a positive integer $n$ such that $|u|<n$ for each path $u\in\operatorname{Path}(E)$. However, a topological space $(G(E),\tau_{\mathcal{F}_{cf}})$ is discrete if and only if $G(E)$ is finite. Since the filter $\mathcal{F}_{\omega}$ has a countable base a topological space $(G(E),\tau_{\mathcal{F}_{\omega}})$ is always metrizable. However, the space $(G(E),\tau_{\mathcal{F}_{cf}})$ is metrizable iff the set $\op{Path}(E)$ is countable. On the other hand, topological spaces $(\mathcal{P}_2,\tau_{\mathcal{F}_{\omega}})$ and $(\mathcal{P}_2,\tau_{\mathcal{F}_{cf}})$ are homeomorphic.
\end{remark}



\begin{theorem}\label{min} For each graph $E$, $\tau_{\mathcal{F}_{cf}}$ is the coarsest inverse semigroup
topology on $G(E)$.
\end{theorem}

\begin{proof}
Suppose that $(G(E),\tau)$ is a topological inverse GIS over a graph $E$. By $E(G(E))$ we denote the subset of idempotents of the semigroup $G(E)$.
Observe that for any topological inverse semigroup $S$ the maps $\varphi\colon S\to E(S)$, $\varphi(x)=xx^{-1}$, and $\psi\colon S\to E(S)$, $\psi(x)=x^{-1}x$, are continuous. Recall that each idempotent of the semigroup $G(E)$ is of the form $uu^{-1}$ for some path $u\in \operatorname{Path}(E)$. Fix an arbitrary cofinite subset $F$ of $\op{Path}(E)$. Put $H=\{uu^{-1}\in E(G(E))\mid u\in F\}\cup\{0\}$. Since $H$ is a cofinite subset of $E(G(E))$ it is open in $E(G(E))$. Then the set
$$U_F=\{ab^{-1}\in G(E): a,b\in F\}\cup\{0\}=\varphi^{-1}(H)\cap\psi^{-1}(H)$$ is an open neighbourhood of $0$ in the topological space $(G(E),\tau)$ which yields the inclusion $\tau_{cf}\subset\tau$.
\end{proof}

\begin{proposition}\label{main1}
For an arbitrary graph $E$ the topological inverse semigroup $(G(E),\tau_{\mathcal{F}_{\omega}})$ embeds into the topological inverse polycyclic monoid $(\mathcal{P}_{|G(E)|},\tau_{\mathcal{F}_{\omega}})$.
\end{proposition}
\begin{proof}
Observe that for a map $F$ defined in Theorem~\ref{theorem1} the following equation holds: $F(U_n(0))=\mathcal{P}_{|G(E)|}\cap V_{n+1}(0)$ where $U_{n}(0)=\{uv^{-1}\in G(E)\mid \min\{|u|,|v|\}>n\}$ is an open neighbourhood of $0$ in $(G(E),\tau_{\mathcal{F}_{\omega}})$ and $V_{n+1}=\{uv^{-1}\in \mathcal{P}_{|G(E)|}\mid \min\{|u|,|v|\}>n+1\}$ is an open neighborhood of $0$ in $(\mathcal{P}_{|G(E)|},\tau_{\mathcal{F}_{\omega}})$. Hence $F$ is a topological embedding.
\end{proof}

\begin{remark}
Despite the fact that monoid $\mathcal{P}_{\omega}$ embeds into the monoid $\mathcal{P}_{2}$, the topological inverse monoid $(\mathcal{P}_{\omega},\tau_{\mathcal{F}_{\omega}})$ does not embed into $(\mathcal{P}_{2},\tau_{\mathcal{F}_{\omega}})$. Observe that the semilattice $E(\mathcal{P}_{2})$ is a compact subspace of  $(\mathcal{P}_{2},\tau_{\mathcal{F}_{\omega}})$. Recall that all non-zero elements in $(\mathcal{P}_{2},\tau_{\mathcal{F}_{\omega}})$ are isolated. Then for each  homomorphism $f:(\mathcal{P}_{\omega},\tau_{\mathcal{F}_{\omega}})\rightarrow (\mathcal{P}_{2},\tau_{\mathcal{F}_{\omega}})$ the image $f(E(\mathcal{P}_{\omega}))$ is closed and hence compact subset of $E(\mathcal{P}_{2})$. But $E(\mathcal{P}_{\omega})$ is not compact in $(\mathcal{P}_{\omega},\tau_{\mathcal{F}_{\omega}})$.
\end{remark}

However, the situation is different in the case of the topology $\tau_{\mathcal{F}_{cf}}$.

\begin{theorem}\label{main2}
For an arbitrary graph $E$ each injective homomorphism $f:(G(E),\tau_{\mathcal{F}_{cf}})\rightarrow (\mathcal{P}_{\lambda},\tau_{\mathcal{F}_{cf}})$ is a topological embedding.
\end{theorem}
\begin{proof}
Let $E$ be an arbitrary graph and $\lambda$ be any cardinal such that there exists an injective homomorphism $f:G(E)\rightarrow \mathcal{P}_{\lambda}$ (by Theorem~\ref{theorem1}, such cardinal $\lambda$ exists). By $\hat{E}$ we denote the graph which consists of one vertex and $\lambda$ distinct loops.  Observe that semigroups $G(\hat{E})$ and $\mathcal{P}_{\lambda}$ are isomorphic. If graph inverse semigroup $G(E)$ is finite, then the proof is straightforward. Suppose that $G(E)$ is infinite.

We claim that $f(0)=0$. Assuming the contrary, let $f(0)=uu^{-1}\in E(\mathcal{P}_{\lambda})\setminus\{0\}$. Then for each idempotent $e\in E(G(E))$, $f(e)\in\{vv^{-1}\mid u\leq v\}$. The injectivity of the map $f$  implies that the set $E(G(E))$ is finite. Since each GIS is combinatorial (each $\mathcal{H}$-class is singleton) we obtain that semigroup $G(E)$ is finite (the injection $h:G(E)\rightarrow E(G(E))\times E(G(E))$ can be defined as follows: $h(uv^{-1})=(uu^{-1},vv^{-1})$ and $h(0)=(0,0)$) which contradicts our assumption. Hence $f(0)=0$.

Fix an arbitrary open neighbourhood $U_F(0)$ of $0$ in $(\mathcal{P}_{\lambda},\tau_{\mathcal{F}_{cf}})$. Put $A=\{uu^{-1}\mid u\in F\}.$  Denote $H=\{a\in \op{Path}(E)\mid aa^{-1}\in f^{-1}(A)\}$. Since the map $f$ is injective $H$ is a cofinite subset of $\op{Path}(E)$. We claim that $f(U_H(0))\subseteq f(G(E))\cap U_F(0)$. Indeed, fix an arbitrary element $ab^{-1}\in U_H(0)$ and put $f(ab^{-1})=cd^{-1}$. Observe that $$cc^{-1}=cd^{-1}dc^{-1}=f(ab^{-1})\cdot f(ba^{-1})=f(ab^{-1}\cdot ba^{-1})=f(aa^{-1}).$$
The definition of the set $H$ implies that $c\in F$. Analogously, $$dd^{-1}=dc^{-1}\cdot cd^{-1}=f(ba^{-1})\cdot f(ab^{-1})=f(bb^{-1})$$ which implies that $d\in F$. Then $cd^{-1}=f(ab^{-1})\in U_{F}(0)$. Hence $f(U_{H}(0))\subseteq f(G(E))\cap U_{F}(0)$ which provides a continuity of the map $f$.

Fix an arbitrary cofinite subset $H$ of $\op{Path}(E)$. Put $\op{Path}(E)\setminus H=\{a_1,\ldots,a_n\}$. For each $i\leq n$ denote $f(a_i)=u_iv_i^{-1}$. Put $G=\op{Path}(\hat{E})\setminus(\{u_i\}_{i\leq n}\cup\{v_i\}_{i\leq n})$.
Obviously, $G$ is a cofinite subset of $\op{Path}(\hat{E})$. Fix an arbitrary element $ab^{-1}\notin U_{H}(0)$. Then $a\notin H$ or $b\notin H$. Suppose that $a\notin H$. Since $r(a)=r(b)$ we obtain that $a^{-1}a=b^{-1}b$. By our assumption, there exists $i\leq n$ such that $f(a)=u_iv_i^{-1}$. Then $f(b)=uv_i^{-1}$, because $f(a^{-1}a)=v_iv_i^{-1}=f(b^{-1}b)$. Hence $f(ab^{-1})=f(a)\cdot f(b)^{-1}=u_iu^{-1}\notin U_{G}(0)$. Similar arguments work when $b\notin H$ which provides that $f(ab^{-1})\notin U_G(0)$. Hence $U_{G}(0)\cap f(G(E))\subseteq f(U_H(0))$ which implies that the homomorphism $f$ is open on the image $f(G(E))$.
Hence the map $f$ is a topological embedding.
\end{proof}

However, in general case we have a very intriguing question:
\begin{question}
Is it true that each topological inverse graph inverse semigroup $G(E)$ is a subsemigroup of a topological inverse polycyclic monoid $\mathcal{P}_{|G(E)|}$?
\end{question}


\begin{thebibliography}{00}

\bibitem{Abrams-2005}
G.~Abrams, G. Aranda Pino,
{\em The Leavitt path algebra of a graph},
J. Algebra \textbf{293} (2005), 319--334.

\bibitem{Amal-2016}
Amal AlAli, N.D. Gilbert,
{\em Closed inverse subsemigroups of graph inverse semigroups},
arXiv:1608.04538.



\bibitem{Ara-2007}
P.~Ara, M.~A.~Moreno, E.~Pardo,
{\em Non-stable K-theory for graph algebras},
Algebr. Represent. Theory \textbf{10} (2007), 157--178.









\bibitem{Bardyla-2018(1)}
S.~Bardyla,
\emph{An alternative look at the structure of graph inverse semigroups},
preprint, arXiv:1806.09671.


\bibitem{Bardyla-2018}
S.~Bardyla,
\emph{On locally compact semitopological graph inverse semigroups},
Matematychni Studii. \textbf{49} (1), (2018), 19--28.
	

\bibitem{Bardyla-2017}
S.~Bardyla,
\emph{On locally compact topological graph inverse semigroups},
preprint, (2017), arXiv:1706.08594.



\bibitem{Bardyla-2016(1)}
S.~Bardyla,
\emph{Classifying locally compact semitopological polycyclic monoids},
Math. Bulletin of the Shevchenko Scientific Society, {\bf 13}, (2016), 21--28.



\bibitem{BardGut-2016(1)}
S.~Bardyla, O.~Gutik,
\emph{On a semitopological polycyclic monoid},
Algebra Discr. Math. \textbf{21} (2016), no. 2, 163--183.

\bibitem{BardGut-2016(2)}
S.~Bardyla, O.~Gutik,
\emph{On a complete topological inverse polycyclic monoid},
Carpathian Math. Publ. \textbf{8} (2), (2016), 183--194.







\bibitem{Clifford-Preston-1961-1967}
A.~Clifford, G.~Preston,
\emph{The Algebraic Theory of Semigroups}, Vols. I and II,
Amer. Math. Soc. Surveys {\bf 7}, Providence, R.I.,  1961 and  1967.

\bibitem{Cuntz-1980}
J.~Cuntz, W.~Krieger,
{\em A class of $C^{*}$-algebras and topological Markov chains},
Invent. Math. \textbf{56} (1980), 251--268.






\bibitem{Gutik-2015}
O.~Gutik,
\emph{ On the dichotomy of the locally compact semitopological bicyclic monoid with adjoined zero}.
Visn. L'viv. Univ., Ser. Mekh.-Mat. \textbf{80} (2015), 33--41.














\bibitem{Jones-2011}
D.~Jones,
\emph{ Polycyclic monoids and their generalizations}.
PhD Thesis, Heriot-Watt University, 2011.

\bibitem{Jones-Lawson-2014}
D.~Jones, M.~Lawson,
\emph{Graph inverse semigroups: Their characterization and completion},
J. Algebra \textbf{409} (2014), 444--473.

\bibitem{Kumjian-1998}
A.~Kumjian, D.~Pask, I.~Raeburn,
\emph{Cuntz-Krieger algebras of directed graphs},
Pacific J. Math. \textbf{184} (1998), 161--174.


\bibitem{Lawson-1998}
M.~Lawson,
\emph{Inverse Semigroups. The Theory of Partial Symmetries},
Singapore: World Scientific, 1998.

\bibitem{Lawson-2009}
M.~Lawson,
\emph{Primitive partial permutation representations of the polycyclic monoids and branching function systems},
Period. Math. Hungar. \textbf{58} (2009), 189--207.

\bibitem{Meakin-1993}
J.~Meakin, M.~Sapir,
\emph{Congruences on free monoids and submonoids of polycyclic monoids},
J. Austral. Math. Soc. Ser. A \textbf{54} (2009), 236--253.

\bibitem{Mesyan-2016}
Z.~Mesyan, J.D.~Mitchell,
{\em The structure of a graph inverse semigroup},
Semigroup Forum \textbf{93} (2016), 111-130.

\bibitem{Mesyan-Mitchell-Morayne-Peresse-2013}
Z.~Mesyan, J. D.~Mitchell, M.~Morayne, Y. H. P\'{e}resse,
\emph{Topological graph inverse semigroups},
Topology and its Applications, Volume \textbf{208} (2016), 106–126.


\bibitem{Nivat-Perrot-1970}
M.~Nivat, J.-F.~Perrot,
\emph{Une g\'{e}n\'{e}ralisation du monoide bicyclique},
C. R. Acad. Sci., Paris, S\'{e}r. A \textbf{271} (1970), 824--827.


\bibitem{Paterson-1999}
A.~Paterson,
\emph{Graph inverse semigroups, groupoids and their $C^{*}$-algebras},
Birkh\"auser, 1999.






\end{thebibliography}
\end{document}